\documentclass[11pt, a4paper]{amsart}
\usepackage{amsfonts,amssymb,amsmath, amsthm}
\usepackage{mathrsfs, mathtools}
\usepackage{lmodern}
\usepackage{qsymbols}
\usepackage{latexsym}
\usepackage[noadjust]{cite}
\usepackage{pdfsync}
\usepackage{bm}
\usepackage{enumitem}

\newtheorem{thm}{Theorem}[section]
\newtheorem{lem}[thm]{Lemma}
\newtheorem{prop}[thm]{Proposition}
\newtheorem{cor}[thm]{Corollary}

\theoremstyle{definition}
\newtheorem{defn}[thm]{Definition}
\newtheorem{rem}[thm]{Remark}

\numberwithin{equation}{section}
\usepackage[plainpages=false,pdfpagelabels,backref=page,citecolor=red]{hyperref}
\usepackage{xcolor}
\hypersetup{
colorlinks,
linkcolor={cyan!90!black},
citecolor={magenta},
urlcolor={green!40!black}
} 
\newcommand{\R}{\mathbb{R}}

\newcommand{\dist}{\operatorname{dist}}
\newcommand{\loc}{\operatorname{loc}}
\renewcommand{\L}{\operatorname{L}} 
\renewcommand{\H}{\operatorname{H}}
\renewcommand{\H}{\operatorname{H}} 

\renewcommand{\d}{\, \mathrm{d}} 
\newcommand{\tor}{\textup{tor}}
\renewcommand{\Cap}{\text{cap}}

\renewcommand{\limsup}
{\textup{limsup}}

\def\Xint#1{\mathchoice
{\XXint\displaystyle\textstyle{#1}}%
{\XXint\textstyle\scriptstyle{#1}}%
{\XXint\scriptstyle\scriptscriptstyle{#1}}%
{\XXint\scriptscriptstyle%
\scriptscriptstyle{#1}}%
\!\int}
\def\XXint#1#2#3{{\setbox0=\hbox{$#1{#2#3}{%
\int}$ }
\vcenter{\hbox{$#2#3$ }}\kern-.6\wd0}}
\def\overlineint{\,\Xint -} 
\def\overlineiint{\overlineint_{} \kern-.4em \overlineint}
\def\overlineiiint{\overlineiint_{} \kern-.4em \overlineint}
\renewcommand{\iint}{\int_{}\kern-.34em \int} 
\renewcommand{\iiint}{\iint_{}\kern-.34em \int} 

\title{A comparison method for the fractional Laplacian and applications}

\author{Alireza Ataei}
\email{alireza.ataei@math.uu.se}
\address{Department of Mathematics, Uppsala University, S-751 06 Uppsala,
Sweden}

\author{Alireza Tavakoli}
\email{alirezat@kth.se}
\address{Department of Mathematics, KTH Royal Institute of Technology, Stockholm, Sweden}



\date{\today}

\begin{document}

\maketitle

\begin{abstract}
    We study the boundary behavior of solutions to fractional Laplacian. As the first result, the isolation of the first eigenvalue of the fractional Lane-Emden equation is proved in the bounded open sets with Wiener regular boundary. Then, a generalized Hopf's lemma and a global boundary Harnack inequality are proved for the fractional Laplacian.
\end{abstract}

\section{Introduction}
The starting point of our investigation was the problem of isolation of the first eigenvalue of the fractional Lane-Emden equation. Let $1<q <2$, $0<s<1$, and $(-\Delta)^s$ be the fractional $s$-Laplacian. We consider the following minimization problem
\begin{equation}\label{eq:fist-eignevalue}
\begin{aligned}
\Lambda_q:=\inf_{\phi \in C^{\infty}_0(\Omega)} \Bigl\lbrace \iint_{\R^n \times \R^n} \frac{|\phi(x)-\phi(y)|^2 }{|x-y|^{n+2s}} \, \d x \d y \, : \|\phi\|_{\L^q(\Omega)}=1 \Bigr\rbrace,
\end{aligned}
\end{equation}
for a bounded open set $\Omega \subset \R^n$ and define $\mathcal{D}^{s,2}_0(\Omega)$ as the completion of $C^{\infty}_0(\Omega)$ with respect to the semi-norm
\begin{align*}
    [\phi]_{\H^s(\R^n)}:= \biggl(\iint_{\R^n \times \R^n}\frac{|\phi(x)-\phi(y)|^2 }{|x-y|^{n+2s}} \, \d x \d y\biggr)^{1/2}, \quad \phi \in C^{\infty}_0(\R^n).
\end{align*}
We say that $\Omega$ has Wiener regular boundary for $s$-fractional Laplacian if weak solutions of $(-\Delta)^s u= f$ in $\Omega$, where $u \in C(\R^n \setminus \Omega)$ and $f \in \L^{\infty}(\Omega)$, belong to $C(\R^n)$, see Definition \ref{def:regularbound} for more details.

In the first part of this work, we prove that the first eigenvalue $\Lambda_q$ is isolated. 
\begin{thm}
\label{thm:isolation}
Assume that $\Omega \subset \R^n$ has Wiener regular boundary for $s$-fractional Laplacian. Then, there are no sequences $\lambda_i \in \R^+  \setminus \{\Lambda_q\}, v_i \in \mathcal{D}^{s,2}_0(\Omega)$ such that 
   \begin{equation}
\begin{aligned}
\|v_i\|_{\L^q(\Omega)}&= 1,\\
(-\Delta)^s v_i &= \lambda_i |v_i|^{q-2} v_i, \quad \text{in } \Omega,\\
 \lim_{i \to \infty} \lambda_i &= \Lambda_q.\\
\end{aligned}   
   \end{equation}
\end{thm} 

The isolation property for the fractional Lane-Emden equation has already been proved in \cite{FL} for $C^{1,1}$ domains. The approach there is to study the compactness of the embedding $\mathcal{D}^{s,2}_0(\Omega) $ into a weighted space $ \L^2(\Omega;w)$.  The major difficulty with the isolation is the boundary regularity of the positive solutions for $1<q<2$. In fact, Hopf's type lemmas have been used (or assumed) in all the works concerning isolation of the first eigenvalue for the range $1<q<2$, see \cite{BDF,FL}. In the case of local Laplacian, in \cite{BDF} the compactness of this embedding has been studied in detail, and a sharp sufficient condition has been obtained for bounded Lipschitz domains in terms of a uniform interior cone condition. However, for the fractional Laplacian, one needs a $C^{1,1}$ condition on the boundary to obtain Hopf's lemma. We generalize the result for the fractional case to open sets with rougher boundaries such as Lipschitz boundaries. A different method is used to get this generalization, which is simpler than previous ones. We also provide a very short proof for the isolation of the first eigenvalue for the case of $C^{1,1}$ domains, which does not use the weighted compact embedding, see Remark \ref{rmk:easyproof}.  Note that the simplicity of the first eigenvalue $\Lambda_q$, as well as the nonlinearity of the equation, motivates us to restrict ourselves to $1<q<2$.

In application, the isolation of the first eigenvalue of the nonlocal Lane-Emden equation is relevant to the study of the long-time asymptotic behavior of the fractional porous-medium equation, see \cite{FV}.

The second part is motivated by the study of fractional Reaction-Diffusion problems. Let $u$ be a non-negative solution of $(-\Delta)^s u = f(u)$ in a bounded set $\Omega \subset \R^n$, where $f(0)=0.$ The dead core, the set $\{u=0\},$ determines if the reaction takes place. Hence, it is interesting to prove that $u$ is strictly positive on $\Omega$.

There are two main results in this section. First, we prove that a generalized Hopf's lemma holds for viscosity supersolutions for fractional Laplacian. To bring the main results, let $\Omega \subset \R^n$ be a bounded open set and $0<s<1$. Notice that the positive minimizer of \eqref{eq:fist-eignevalue} associated with $\Lambda_1$ say $\Phi_1$ satisfies the following equation
\[
\begin{aligned}
    (- \Delta)^s \Phi_1 = \Lambda_1, \quad & \text{in } \Omega, \\
    \Phi_1 =0, \quad &\text{in } \R^n \setminus \Omega.
\end{aligned}
\]
After a rescaling, the equation reduces to the following problem
\begin{align*}
    (-\Delta)^s u_{\tor} = 1,\quad &\text{in } \Omega,\\
    u_{\tor}=0, \quad &\text{in } \R^n \setminus \Omega.
\end{align*}
Define $K \Subset \Omega$, for a subset $K \subset \Omega$, if $\overline{K} \subset \Omega$. We prove the following generalization of Hopf's lemma.
\begin{lem}
\label{lem:Hopf's lemma}
    Let $u \in \L^1_{2s}(\R^n) \cap C(\overline \Omega)$ be a non-negative function and $K \Subset \Omega.$ Assume that $(-\Delta)^s u \geq f$ in $\Omega$ in the viscosity sense, where $f \in C(\Omega)$ satisfies 
\begin{align*}
    \textup{limsup}_{x \to x_0} f(x) \geq -2\int_{K} \frac{u(y)}{|x_0 - y|^{n+2s}} \, \d y, \quad & \text{if } x_0 \in \partial \Omega, \lim_{x \to x_0} u(x) =0,\\
    f(x_0) > -2 \int_{\R^n} \frac{u(y)}{|x_0-y|^{n+2s}} \, \d y, \quad &\text{if } x_0 \in \Omega, \, u(x_0)=0.
\end{align*}
Then, $u>0$ in $\Omega$ and
\begin{align*}
   u \geq C u_{\tor}, \quad \text{in } \Omega,
\end{align*}
    for a constant $C>0.$
\end{lem}
 Our particular interest is when $f = g(u) + h$ for functions $g \in C([0,\infty)), h \geq 0.$
\begin{cor}
\label{cor:semilinearhopf}
      Let $u \in \L^1_{2s}(\R^n) \cap C(\overline \Omega)$ be non-negative and $g \in C([0,\infty))$. Assume that $g(0)=0$ and $$(-\Delta)^s u \geq g(u), \quad \text{in } \Omega,$$ in the viscosity sense. Then, either $u=0$ a.e. in $\R^n$ or $u>0$ in $\Omega$ and $
   u \geq C u_{\tor}$ in  $\Omega$ for a constant $C>0.$
\end{cor}

For $C^{1,1}$ domains, we have  $u_{\tor} \geq C d(x,\partial \Omega)^s$ for some $C>0,$ see \cite[Lem. 7.3]{R}. Hence, Corollary \ref{cor:semilinearhopf} generalizes Hopf's lemma for $C^{1,1}$ domains, see \cite{CRS,GS,R}. See also \cite{AFT} in which similar ideas as in \cite{GS} and the current article has been explored to prove a Hopf lemma in the setting of the regional fractional Laplacian for $C^{1,1}$ domains. We also refer to \cite[Thm. 2.6]{IMS} for the $s-$fractional $p-$Laplacian case with the extra assumption that $g \in BV_{\loc}(\R)$. Note that Corollary \ref{cor:semilinearhopf} does not hold for the local Laplacian. There are several works which determine the necessary and sufficient conditions to have Hopf's lemma and strong maximal principle for local quasilinear Laplacian, see for instance \cite{V,PS}. To explain the contrast between them and the nonlocal case, let $-\Delta u \geq g(u)$ weakly 
in $\Omega$ for a nonzero function $u \in C^1(\Omega)$, satisfying $u \geq 0$ in $\Omega$, and $g \in C([0,\infty])$ be a non-increasing function, satisfying $g(0)=0$. Then, it has been shown in \cite{V} that $\int_{0}^{1} \frac{1}{\sqrt{-s g(s)}} \, \d s = \infty$ is necessary and sufficient to have $u>0$ in $\Omega.$ For example, the function $g(s) = -s^p$ for $0<p<1$ and all $s \in [0,\infty)$ satisfies the conditions in Corollary \ref{cor:semilinearhopf} but $\int_{0}^1 \frac{1}{\sqrt{-s g(s)}} \, \d s = \frac{2}{p+1}.$ Another interesting point is that we do not need the function $g$ to be a non-increasing function. Hence, it may change sign around $0$. 

Now, it is unclear how the constants in Lemma \ref{lem:Hopf's lemma} and Corollary \ref{cor:semilinearhopf} are related to $u$. To resolve the dependence of the constants and obtain an upper bound, we have the following result which is comparable to boundary Harnack principles with nonzero right-hand sides see for example \cite{BFV,RS,AH,RT}.

\begin{thm}
\label{thm:boundaryharnack}
    Let $u,v \in C(\overline \Omega) \cap \L^{\infty}(\R^n)$ satisfy  
    \begin{align*}
    u>0 , \,v>0 \quad & \text{in } \; \Omega,\\
0\leq  u = v \leq 1  \quad & \text{in }  \R^n \setminus \Omega,
    \end{align*}
 and
    \begin{align*}
   - 2 (\mathrm{diam}\, \Omega)^{-(n+2s)}\int_{K} u(y)\, \d y \leq (-\Delta)^s u \leq  1, \quad & \text{in } \Omega,\\
    -2 (\mathrm{diam}\, \Omega)^{-(n+2s)}\int_{K} v(y) \, \d y  \leq   (-\Delta)^s v \leq 1, \quad &\text{ in } \Omega,
    \end{align*}
    in the viscosity sense, where $K \Subset \Omega$. Assume that either $u(x_0)\geq D , \; v(x_0)\geq D$ for a fixed point $x_0 \in \Omega \setminus K$ or $\|u\|_{\L^p(\Omega \setminus K)} \geq D , \;  \|v\|_{\L^p(\Omega \setminus K)}\geq D$ for $D>0, 1 \leq p < \infty.$ Then,
    \begin{align*}
       C_1 \leq \frac{u}{v} \leq C_2, \quad \text{in } \Omega,
    \end{align*} where $C_1, C_2$ are positive constants depending on $\Omega, K, n, s, D, x_0$ or $p$.
\end{thm}

Here, no assumption on the boundary is required except that the solutions should be continuous at the boundary, similar to the result in \cite[Thm. 1.2]{RS}. On the other hand, we do not need a smallness condition on the upper bound for the fractional Laplacian, which is more general in comparison to the result in \cite[Thm. 1.2]{RS}. If one compares this result to the boundary Harnack principle for the local Laplacian, then the assumptions on the right-hand sides in Theorem \ref{thm:boundaryharnack} are in accordance with the assumptions in \cite{AH}.

The work can be summarized as follows: We bring the definition of the required spaces and notions of solutions in Section \ref{sec3}. Then, in Section \ref{sec4}, we discuss the continuity at the boundary and the Wiener criterion for the $s$-fractional Laplacian with a nonzero right-hand side. In Section \ref{sec5}, we focus on the isolation of the fractional Lane-Emden equation. We give a pointwise-based proof for the isolation of the first eigenfunction. In Section \ref{sec6}, we prove Hopf's lemma and a global version of boundary Harnack inequality for fractional Laplacian.

In this work, we avoid the study of fractional equations with rough elliptic kernels due to the simplicity and clarity of the proofs. However, one can generalize the proofs for rough kernels using the same method as we provide here.

Finally, we acknowledge that the idea of the proof of the main results is motivated by the amazing proof of \cite[Lem. 3.1]{GS}.

\section{Acknowledgement}
 We thank Giovanni Franzina for his interesting talk at Mittag-Leffler and his comments on our early ideas, which draw our attention to the connection between boundary Harnack and the isolation of the first eigenvalue. We also thank Erik Lindgren for several comments and helpful corrections. In addition, we acknowledge several helpful discussions with Kaj Nyström and Henrik Shahgholian regarding boundary Harnack inequality, which has increased our knowledge for the current work.

This work was done while the authors were participating in the program geometric aspects of nonlinear partial differential equations at Mittag-Leffler institute in Djursholm, Sweden during the fall of 2022. The research program is supported by Swedish Research Council grant no. 2016-06596.
A.A. would like to thank Uppsala university for covering the travel costs for attending this program. A.T.  has partially been supported by the Swedish Research Council, grant no. 2017- 03736. During the development of this paper, A.T. has been a PhD student at Uppsala University. In particular, he wishes
to express his gratitude to the Department of Mathematics at Uppsala University for its warm and hospitable
research environment.
\section{Preliminaries}
\label{sec3}
In the entire work, $0 <s <1$ and $\Omega \subset \R^n$ is a bounded open set for $n \geq 1$. 
\subsection{Function spaces}
The space $\L^1_{\loc}(\Omega)$ consists of all the measurable real-valued functions on $\R^n$, which are locally integrable functions on $\Omega$. We define $\L^1_{2s}(\R^n)$ as the space of $u \in \L^1_{\loc}(\R^n)$ with the finite norm
$$\|u\|_{\L_{2s}^1(\R^n)} := \int_{\R^n} \frac{|u(x)|}{1+|x|^{n+2s}} \,\d x.$$ Denote $\mathcal{D}^{s,2}_0(\Omega)$ as the homogeneous fractional Sobolev space obtained by the completion of $C^{\infty}_0(\Omega)$ with respect to the seminorm 

$$[u]_{\H^s(\R^n)}^2 := \iint_{\R^n \times \R^n} \frac{|u(x)-u(y)|^2}{|x-y|^{n+2s}} \, \d x \d y,$$
and $\H^s(\Omega)$ as the space of measurable functions $u: \R^n \to \R$ with the finite norm
\begin{align}
   \| u\|^2_{\H^s(\Omega)} := \int_{\Omega} |u(x)|^2 \, \d x + \iint_{\Omega \times \Omega} \frac{|u(x)-u(y)|^2}{|x-y|^{n+2s}} \, \d x \d y.
\end{align}
We say that $u \in \H^s_{\loc}(\Omega)$ if 

\begin{align*}
   \int_{K} |u(x)|^2 \, \d x + \iint_{K \times K} \frac{|u(x)-u(y)|^2}{|x-y|^{n+2s}} \, \d x \d y < \infty
\end{align*}
for every compact set $K \subset \Omega.$
For the boundary value Dirichlet problem, we need the following space of Sobolev-type
\[
V^{s,2}(\Omega|\R^n):=\left\lbrace u:\R^n \to \R \; : \; u|_{\Omega} \in \L^2(\Omega) \; , \; \frac{u(x)-u(y)}{|x-y|^{n/2+s}} \in \L^2(\Omega \times \R^n)  \right\rbrace,
\]
equipped with the norm
$$\|u\|^2_{V^{s,2}(\Omega|\R^n)} := \|u\|^2_{\L^2(\Omega)} + \iint_{\Omega\times \R^n} \frac{|u(x)-u(y)|^2}{|x-y|^{n+2s}} \d x \d y.$$
We consider the space of Sobolev functions taking a boundary value $0$ by
\[
V_0^{s,2}(\Omega):= \overline{C^\infty_0(\Omega)}^{V^{s,2}(\Omega|\R^n)}.
\]
By the fractional Poincar\'{e}-Sobolev inequality, see Theorem \ref{thm:sobolevpoincare-ineq}, we obtain
$$V_0^{s,2}(\Omega)= \mathcal{D}_{0}^{s,2}(\Omega).$$
Finally, we define the space of functions taking boundary $g \in V^{s,2}(\Omega|\R^n)$ by 
$$V_g^{s,2}(\Omega|\R^n):= \left\lbrace u\in V^{s,2}(\Omega|\R^n) \; : \; u-g \in V_0^{s,2}(\Omega) \right\rbrace.$$
The space of $C(K)$ for $K \subset \R^n$ denotes all the measurable functions $u: \R^n \to \R$ which are continuous on $K$. For every $\alpha>0$, the space $C^{\alpha}_{\loc}(\Omega)$ includes measurable functions $u: \R^n \to \R $ such that the function $ (x,y) \to 
    \frac{|u(x)-u(y)|}{|x-y|^{\alpha}} 
$ is locally bounded on $ \Omega \times \Omega.$ We say that $u \in C^{1,\alpha}_{\loc}(\Omega) $ if each component of $\nabla u$ belongs to $C^{\alpha}_{\loc}(\Omega)$.

\subsection{Fractional Poincar\'{e}-Sobolev  inequality and fractional Sobolev embedding}
The fractional Sobolev exponent $2^{\ast}_s$ is defined by 
\begin{align*}
2^{\ast}_s := \frac{2n}{n-2s}, \quad & \text{if } 2s < n,\\
2^{\ast}_s := \infty, \quad & \text{if } 2s \geq n.
\end{align*} The following theorem is the fractional Poincar\'{e}-Sobolev  inequality, see \cite[Thm 8.1, 8.2, 9.1]{P}.

\begin{thm}\label{thm:sobolevpoincare-ineq}
Let $\Omega \subset \R^n$ be an open bounded set. Then, for every $u \in C^{\infty}_0(\Omega)$, we have 
\[
\begin{aligned}
    \| u \|_{\L^{2_s^\ast}(\Omega)}^2 \leq C_1 [u]_{\H^s(\R^n)}^2, \quad &\text{ if } 2s < n, \\
    \| u\|_{\L^\infty(\Omega)}^2 \leq C_2 
    [u]_{\H^s(\R^n)}^2, \quad &\text{ if } 2s > n,\\
    \| u\|_{\L^p(\Omega)}^2 \leq C_3 
    [u]_{\H^s(\R^n)}^2, \quad &\text{ if } 2s = n,
\end{aligned}
\]
for every  $1\leq p< \infty$, where the constants $C_1$ depends on $n,s$, $C_2$ depends on $n,s, \Omega$, and $C_3$ depends on $n,s,p, \Omega$.
\end{thm}
In the light of the inequalities above the infimum defined in \eqref{eq:fist-eignevalue} is non-zero. Indeed, for every $1 \leq q \leq  2_s^{\ast} $ (except for the case $n=2s$ ), we have 
$$\| u \|_{\L^q(\Omega)}^2 \leq C [u]_{\H^s(\R^n)}^2 \quad \text{for } u \in \mathcal{D}_0^{s,2}(\Omega),$$
where we used Theorem \ref{thm:sobolevpoincare-ineq} and H\"older's inequality.
In fact, the inverse of the first eigenvalue $\Lambda_q^{-1}$ is the best constant $C$ in the above Sobolev-Poincar\'e type inequality.
\begin{thm}
\label{thm:fractionalsobolev}
   The space  $\mathcal{D}^{s,2}_0(\Omega)$ is compactly embedded in $\L^{q}(\Omega)$ for bounded open sets $\Omega \subset \R^n$ and $1<
   q\leq 2.$
\end{thm}
\begin{proof}
    We take a large enough ball $B \subset \R^n$ such that $\Omega \subset B$. Then, $\mathcal{D}^{s,2}_0(\Omega) \subset \mathcal{D}^{s,2}_0(B)$ and $B$ is an extension domain, see \cite[Thm. 5.4]{DPV}. Hence, by applying \cite[Thm. 7.1]{DPV} to $\mathcal{D}^{s,2}_0(B)$, we complete the proof.
\end{proof}

\subsection{Notions of solutions}
In this section, we define different notions of solutions for $s$-fractional Laplacian equations. The results are standard,  see for instance \cite{R,SV}, and we bring them for the convenience of the reader. 
\begin{defn}
     We say that $u \in \L^1_{2s}(\R^n)$ is a pointwise solution of
\begin{equation}\label{eq:laplac-f}
    (-\Delta)^s u = f, \quad \text{in }  \Omega ,
\end{equation}
if the following integral is convergent for every $x\in \Omega$
$$(-\Delta)^s u(x) := 2\mathrm{P.V.} \int_{\R^n} \frac{u(x)-u(y)}{|x-y|^{n+2s}}\, \d x := 2\lim_{\epsilon \to 0} \int_{\R^n \setminus B(x,\epsilon)} \frac{u(x)-u(y)}{|x-y|^{n+2s}}\, \d x,$$
and equation \eqref{eq:laplac-f} holds pointwise. 
\end{defn}
By abuse of notation, we drop $\mathrm{P.V.}$ in the rest of the paper.

\begin{defn}
For $f \in \L^1_{\loc}(\Omega)$, we say that $u\in \H^s_{\loc}(\Omega)\cap \L_{2s}^1(\R^n)$ is a local weak supersolution (subsolution) of 
$$(-\Delta)^s u \geq (\leq)\, f, \quad \text{in } \Omega,$$
if for any non-negative test function $\phi \in C^\infty_0(\Omega)$ we have
$$\iint_{\R^n \times \R^n}  \frac{(u(x)-u(y))(\phi(x)-\phi(y))}{|x-y|^{n+2s}}\, \d x \d y \geq (\leq)\, \int_{\Omega} f(x) \phi(x)\, \d x.$$
Finally, we say that $(-\Delta)^s u= f$ weakly in $\Omega$ if $(-\Delta)^s u \leq f$ and $(-\Delta)^s u \geq f$ weakly in $\Omega.$
\end{defn}

 Now, we explain the notion of the weak solution to a boundary value problem.
\begin{defn}
    \label{def:weaksol}
   Let $f \in \mathcal{D}^{s,2}_0(\Omega)^\star, g \in V^{s,2}(\Omega|\R^n)$, where $\mathcal{D}^{s,2}_0(\Omega)^\star
   $ is the dual of $\mathcal{D}^{s,2}_0(\Omega) $. We say that a function $u \in V_g^{s,2}(\Omega|\R^n)$ solves the boundary value problem
\begin{equation}\label{eq:weaksol}
\begin{aligned}
(-\Delta)^s u = f \quad &\text{in} \; \Omega, \\
u=g \quad &\text{in}\; \R^n\setminus \Omega,
\end{aligned}
\end{equation}
in the weak sense if
\begin{equation*}
    \iint_{\R^n \times \R^n} \frac{(u(x)-u(y))(\phi(x)-\phi(y))}{|x-y|^{n+2s}}\, \d x \d y = \langle f\,,\, \phi \rangle,
\end{equation*}
for every $\phi \in \mathcal{D}^{s,2}_0(\Omega)$.
\end{defn}
Note that in the case $2s\neq n$, for a function $f\in \L^{({2_s^\ast})^\prime}(\Omega)$, where $({2_s^\ast})^\prime$ is the Hölder dual of ${2_s^\ast}$, the pairing 
$$\langle f\, , \, \phi \rangle:= \int_{\Omega} f(x) \phi(x) \,\d x, $$
defines a continuous functional on $\mathcal{D}_0^{s,2}(\Omega)$, in fact using H\"older's and Sobolev's inequality, for any $\phi \in C^\infty_0(\Omega)$ we have
$$ \left| \int_{\Omega} f(x) \phi(x) \, \d x \right| \leq \| f\|_{\L^{({2_s^\ast})^\prime}(\Omega)} \|\phi \|_{\L^{2_s^\ast}(\Omega)} \leq C \| f\|_{\L^{({2_s^\ast})^\prime}(\Omega)} [\phi]_{\H^s(\R^n)}.$$
\begin{rem}\label{Rmk:classical-weak}
Let $\epsilon >0$ and 
\begin{equation}
\label{eq:regulsol}
\begin{aligned}
    u \in C^{2s+\epsilon}_{\loc}(\Omega)\cap \L_{2s}^1(\R^n), \quad \text{if } 2s<1,\\
    u \in C^{1,2s+\epsilon-1}_{\loc}(\Omega)\cap \L_{2s}^1(\R^n), \quad \text{if } 2s\geq 1.
\end{aligned}
\end{equation}
 Then, the integral 
$$ \frac{1}{2}  \int_{\R^n} \frac{2u(x) -u(x+y)-u(x-y)}{|y|^{n+2s}}\, \d y =  \int_{\R^n} \frac{u(x)-u(y)}{|x-y|^{n+2s}}\, \d y,$$
is convergent in the P.V. sense for every $x \in \Omega$. Also, a weak solution of \eqref{eq:weaksol} satisfying \eqref{eq:regulsol} is a pointwise solution. 
\end{rem}




The final notion of solutions is the viscosity solutions.

\begin{defn}
\label{def:viscosity}
    Let $x \in \Omega$ and $u \in \L_{2s}^1(\R^n)$ be lower semicontinuous in $\Omega$. Then, $ (-\Delta)^s u (x) \leq f(x)$ in $\Omega$ if for every function $\phi \in C^2(U)$, where $U \subset \Omega$ is an open neighbourhood of $x$, that touches $u$ at $x$ from above, i.e., $\phi$ satisfies
    \begin{equation*}
    \begin{aligned}
& \phi(x)=u(x) ,\\
&\phi >u,   \quad \text{in }  U \setminus \{x\}, \end{aligned}
  \end{equation*}
 we have
 \begin{align*}
    2 \int_{\R^n} \frac{w(x)-w(y)}{|x-y|^{n+2s}} \, \d y \leq f(x), 
 \end{align*}
where 
\begin{equation*}
\begin{aligned}
    w := \begin{cases}
        \phi \quad &\text{ in } U,\\
        u(y) \quad &\text{ in } \R^n \setminus U.
    \end{cases}
\end{aligned}
 \end{equation*}
 We define $(-\Delta)^s u (x) \geq f(x)$ if $-u$ is lower semicontinuous on $ \Omega$ and $(-\Delta)^s (-u) (x) \leq -f(x)$. Also, $(-\Delta)^s u(x)=f(x)$ if $u \in C( \Omega)$, $(-\Delta)^s u(x) \geq f(x)$, and $(-\Delta)^s u(x) \leq f(x)$. 
\end{defn} 
In the rest of the work, for brevity, we do not repeat the assumption of $\L^1_{2s}$ and lower (upper) semicontinuous conditions for the viscosity solutions.

The following proposition shows the connection between pointwise and viscosity solutions.
\begin{prop} \label{prop: equipointviscos}
    Let $u \in \L^1_{2s}(\R^n) \cap C( \Omega)$ and $(-\Delta)^s u=f$ pointwise in $\Omega$. Then,
\begin{align*}
    (-\Delta)^s u = f, \quad \text{in } \Omega,
\end{align*} 
in the viscosity sense.
    
\end{prop}
\begin{proof}
 Let $x$ be a fixed point in $\Omega$. Assume that there is a function $\phi \in C^2(B(x,r))$ such that
    \begin{align*}
\phi(x) &= u(x),\\
\phi &> u, \quad \text{in } B(x,r) \setminus \lbrace x \rbrace     
    \end{align*}
    for a fixed $r>0.$
   Define \begin{align*}
        w := \begin{cases}
        \phi \quad &\text{ in } B(x,r),\\
        u(y) \quad &\text{ in } \R^n \setminus B(x,r).\end{cases}
    \end{align*}
Hence,
\begin{align*}
    \int_{\R^n } \frac{u(x)-u(y)}{|x-y|^{n+2s}} \, \d y 
  & =  \int_{\R^n \setminus B(x,r)} \frac{u(x)-u(y)}{|x-y|^{n+2s}} \, \d y +\lim_{\epsilon \to 0}  \int_{B(x,r) \setminus B(x,\epsilon) } \frac{u(x)-u(y)}{|x-y|^{n+2s}} \, \d y\\ &\geq  \int_{\R^n \setminus B(x,r)} \frac{u(x)-u(y)}{|x-y|^{n+2s}} \, \d y +\lim_{\epsilon \to 0} \int_{B(x,r) \setminus B(x,\epsilon) } \frac{\phi(x)-\phi(y)}{|x-y|^{n+2s}} \, \d y \\& =  \int_{\R^n } \frac{w(x)-w(y)}{|x-y|^{n+2s}} \, \d y
\end{align*}
Therefore, $(-\Delta)^s u(x) \leq 2\int_{\R^n} \frac{u(x)-u(y)}{|x-y|^{n+2s}}\, \d y$ in the viscosity sense. 
Likewise, one can prove
$ (-\Delta)^s u(x) \geq   2 \int_{\R^n} \frac{u(x)-u(y)}{|x-y|^{n+2s}}\, \d y$ in the viscosity sense, which completes the proof.

\end{proof}
The next proposition proves that weak solutions are viscosity solutions if the right-hand side is continuous.
\begin{prop}
\label{prop:equaivweakviscos}
Let $u \in \H^{s}_{\loc}(\Omega) \cap \L_{2s}^1(\R^n) \cap C( \Omega)$ and $f \in C(\Omega)$. If $u$ is a local weak supersolution of $(-\Delta)^s u \geq f $ in $\Omega$, then $(-\Delta)^s u \geq f $ in $\Omega$ in the viscosity sense.
\end{prop}
\begin{proof}
    The idea is to approximate $u$ with smooth weak solutions and apply the equivalence of the solutions, which is similar to \cite[Thm. 1]{SV}. 
    Let $0< \epsilon <\frac{1}{4} $ and
$\eta_{\epsilon}(x) := \frac{1}{\epsilon^n} \eta(\frac{x}{\epsilon})$ for $x \in \R^n$, where $\eta \in C^{\infty}_0(B(0,1))$ is a radial function, satisfying $0\leq \eta \leq 1$ and $\int_{\R^n} \eta(x) \, \d x=1$. Since $\eta_{\epsilon} \ast u$ is smooth, it belongs to $\H^s_{\loc}(\Omega).$ Now, we verify that $ \eta_{\epsilon} \ast u \in \L^{1}_{2s}(\R^n)$. In fact,
    \[
    \begin{aligned}
    \| \eta_\epsilon \ast u \|_{\L_{2s}^{1}(\R^n)} &\leq \frac{1}{\epsilon^{n}}\int_{\R^n} \int_{\R^n} \frac{\eta\Big(\frac{y}{\epsilon}\Big) |u(x-y)|}{1+ |x|^{n+2s}} \d y \d x =  \int_{\R^n} \int_{\R^n} \frac{\eta(y) |u(x-\epsilon y)|}{1+ |x|^{n+2s}} \d y \d x \\
    &\leq \int_{B(0,1)} \int_{\R^n} \frac{|u(x-\epsilon y)|}{1+ |x|^{n+2s}} \d x \d y = \int_{B(0,1)} \int_{\R^n} \frac{|u(x)|}{1+ |x + \epsilon y|^{n+2s}} \d x \d y \\
    & \leq 2^{n+2s+1} \int_{B(0,1)} \int_{\R^n} \frac{|u(x)|}{1+ |x|^{n+2s}} \d x \d y \leq 2^{n+2s+1} |B(0,1)| \|u\|_{\L^{1}_{2s}(\R^n)}.
    \end{aligned}
    \]
    Define 
    $\Omega_{\epsilon} := \{x \in \Omega: \dist(x,\partial \Omega) > \epsilon\}$. Then,
    \begin{align*}
      & \iint_{\R^n \times \R^n} \frac{( \eta_{\epsilon} \ast u(x)- \eta_{\epsilon} \ast u(y))(\phi(x)-\phi(y))}{|x-y|^{n+2s}} \d x \d y \\
      &= \int_{\R^n} \left(\iint_{\R^n \times \R^n} \eta_{\epsilon}(z) \frac{ (u(x-z)- u(y-z))(\phi(x)-\phi(y))}{|x-y|^{n+2s}} \d x \d y \right) \d z
      \\&= \int_{\R^n} \iint_{\R^n \times \R^n} \eta_{\epsilon}(z) \frac{ (u(x)- u(y))(\phi(x+z)-\phi(y+z))}{|x-y|^{n+2s}} \d x \d y \d z \\&=   \iint_{\R^n \times \R^n}  \frac{ (u(x)- u(y))\int_{\R^n} \eta_{\epsilon}(z) (\phi(x+z)-\phi(y+z)) \, \d z}{|x-y|^{n+2s}} \d x \d y \\
      &=   \iint_{\R^n \times \R^n}  \frac{ (u(x)- u(y))(\eta_{\epsilon}  \ast \phi(x)-\eta_{\epsilon}  \ast \phi(y))}{|x-y|^{n+2s}} \d x \d y \\
      &\geq \int_{\R^n} f(x) (\eta_{\epsilon} \ast \phi)(x) \, \d x =\int_{\R^n} (\eta_{\epsilon} \ast 
 f)(x) \phi(x) \, \d x 
    \end{align*}
    for every non-negative $\phi \in C^{\infty}_0(\Omega_{\epsilon})$. In conclusion, $(-\Delta)^s \eta_{\epsilon} \ast u \geq \eta_{\epsilon} \ast f$ locally weakly in $\Omega_{\epsilon}$ 
 for every $\epsilon>0$. Since $\eta_{\epsilon} \ast u$ is smooth, by Remark \ref{rmk:easyproof} and Proposition \ref{prop: equipointviscos}, $(-\Delta)^s \eta_{\epsilon} \ast u 
\geq \eta_{\epsilon} \ast f$ in $\Omega_{\epsilon}$ in the viscosity sense. Now, $u_{\epsilon}$ converges pointwise to $u$ a.e. in $\R^n$ and locally uniformly in $\Omega$. Moreover, $\eta_{\epsilon} \ast f$ converges locally uniformly to $f$ in $\Omega.$ Hence, by \cite[Lem. 4.5]{CS2},  $(-\Delta)^s u \geq  f$ in $\Omega$ in the viscosity sense.

\end{proof}

\section{Continuity at the boundary}\label{sec4}
In this section, we provide backgrounds for the Wiener criterion for the $s$-fractional Laplacian with the right-hand side.

A sufficient condition for equations with zero boundary data and smooth compactly supported right-hand side has been established in \cite{E}. The necessary and sufficient condition for continuous compactly supported boundary data is obtained in \cite{Bj} for $s$-harmonic functions. In \cite{KLL}, a necessary and sufficient condition for vanishing right-hand side and continuous boundary data has been established in the context of weak solutions for the more general case of fractional $p$-Laplace type operators. In \cite[A.4.]{Li}, the sufficiency part of the results in \cite{KLL} has been extended to include equations with bounded right-hand sides. In the case of $p=2$, the criterion reads as follows:
\begin{defn}
Define
\begin{align*}
    \Cap_{s,2}(\overline{B(\xi_0,r)} \setminus \Omega,B(\xi_0,2r)) := \inf_v \iint_{\R^n \times \R^n} \frac{|v(x)-v(y)|^2}{|x-y|^{n+2s}} \, \d x \d y,
\end{align*}
    where the infimum is taken over all $v \in C^{\infty}_0(B(\xi_0,2r))$ such that $v \geq 1$ on $\overline{B(\xi_0,r)} \setminus \Omega$ for $\xi_0 \in \partial \Omega, r>0$. We say that a point $\xi_0 \in \partial \Omega$ satisfies the Wiener criterion for fractional linear Laplacian if
    \begin{align*}
        \int_{0}^1 \frac{\Cap_{s,2}(\overline{B(\xi_0,r)} \setminus \Omega,B(\xi_0,2r))}{r^{n-2s}} \, \frac{\d r}{r} = \infty.
    \end{align*}
\end{defn}

We provide details for the sufficient part for general right-hand sides in $\L^\infty$ and $\L^p$ spaces, using a perturbation argument.

First, we define the notion of regular boundaries and the Wiener criterion.
\begin{defn}
\label{def:regularbound}
A point $\xi_0 \in \partial \Omega$ is regular for the $s$-fractional Laplacian if for every $u \in V^{s,2}(\Omega|\R^n) , f \in \L^{\infty}(\Omega), g \in C(\R^n) \cap V^{s,2}(\Omega|\R^n)$, satisfying \eqref{eq:weaksol} in the weak sense, we have $\lim_{\xi \to \xi_0} u(\xi) = g(\xi_0)$. We say that $\Omega$ has Wiener regular boundary for the $s$-fractional Laplacian if all the points on $\partial \Omega$ are regular for the $s$-fractional Laplacian.

\end{defn}

The next lemma provides the existence of weak solutions. We refer to \cite[Thm. 4.4]{FKV} for proof.
\begin{lem}
\label{lem:existenceweaksol}
    Let $f \in \L^q(U)$ for a bounded open set $U \subset \R^n$ with $q \geq \frac{2n}{n+2s}$. Then, there exists a weak solution $u \in \mathcal{D}^{s,2}_0(U)$ satisfying 
    \begin{align*}
        (-\Delta)^s u &= f, \quad \text{in } U.
    \end{align*}
\end{lem}
\begin{prop} \label{prop:Wiener}
    A point $\xi_0 \in \partial \Omega$ is regular for the $s$-fractional Laplacian if and only if it satisfies the Wiener criterion for the $s$-fractional Laplacian.
\end{prop}
\begin{proof}
  If $\xi_0 \in \partial \Omega$ is regular for the $s$-fractional Laplacian, then it satisfies the Wiener criterion for the $s$-fractional Laplacian by \cite[Thm. 1.1]{KLL}. Now, assume that $\xi_0 \in \partial \Omega$  satisfies the Wiener criterion for the $s$-fractional Laplacian. Let $u \in V^{s,2}_{g}(\Omega|\R^n), f \in \L^{\infty}(\Omega), g \in C(\R^n) \cap V^{s,2}(\Omega|\R^n)$, satisfying \eqref{eq:weaksol}. We take a ball $B$ large enough such that $\Omega \subset B.$ By Lemma \ref{lem:existenceweaksol}, there exists a weak solution $w \in \mathcal{D}^{s,2}_0(B)$ of
\begin{align*}
    (- \Delta)^s w = -f, \quad \text{in } B.
\end{align*}
By \cite[Cor. 5.2]{R}, we obtain $w \in \mathcal{D}^{s,2}_0(B) \cap \L^{\infty}(\R^n).$
Since $B$ has a $C^2$ boundary, by \cite[Thm. 2]{SV}, we have $w$ is continuous in $\R^n.$ Define the function $\Tilde{u} := u + w$. Then, $(- \Delta)^s \Tilde{u}=0$ weakly in $\Omega$ and $\Tilde{u}=w +g$ on $\R^n \setminus \Omega$. Hence, by \cite[Thm. 1.1]{KLL} and $w+g \in C(\R^n) \cap V^{s,2}(\Omega|\R^n)$, we obtain $ \lim_{\xi \to \xi_0} \Tilde{u}(\xi) = (w+g)(\xi_0)$. In conclusion, by $w \in C(\R^n)$, it is implied that $\lim_{\xi \to \xi_0} u(\xi) = g(\xi_0)$.

\end{proof}
Finally, we state the following observation.
\begin{prop}
    Let $f\in \L^p(\Omega)$, with $p>\frac{n}{2s}$, and $g \in C(\R^n) \cap V^{s,2}(\Omega|\R^n).$ Then, the unique solution of the Dirichlet problem \eqref{eq:weaksol} attains the boundary value continuously at any regular boundary point for the $s$-fractional Laplacian.
\end{prop}
\begin{proof}
    Assume that $\xi_0 \in \partial \Omega$ is a regular point for the $s$-fractional Laplacian.
    Since $\L^\infty(\Omega)$ is dense in $\L^p(\Omega)$, we can find $f_i \in \L^\infty(\Omega)$ such that
    $$f_i \to f \quad \text{in } \L^p(\Omega).$$
    Now, consider the weak solutions $u_i$ to the following boundary value problems
    \[
    \begin{aligned}
(-\Delta)^s u_i = f_i, \quad &\text{in} \; \Omega, \\
u=g, \quad &\text{in}\; \R^n\setminus \Omega.
\end{aligned}
    \]
    Since $\xi_0$ is a regular point for the $s$-fractional Laplacian, we obtain
    \begin{equation}\label{eq:regular-lim-un}
        \lim_{\xi \to \xi_0} u_i(\xi) = g(\xi_0).
    \end{equation}
 Notice that $u_i-u \in \mathcal{D}^{s,2}_0(\Omega)$ and
 $$(-\Delta)^s (u_i-u) = f_i-f,$$
 weakly in $\Omega.$ Hence, using \cite[Thm. 3.1]{BP}, we have
 $$\| u_i-u\|_{\L^\infty(\Omega)} \leq C \|f_i-f\|_{\L^p(\Omega)},$$
 with $C$ depending on $s,p,|\Omega|$. Thus, we have the uniform convergence of $u_i$ to $u$ in $\Omega$. This uniform convergence, together with \eqref{eq:regular-lim-un}, implies
 $$\lim_{\xi \to \xi_0} u(\xi) = g(\xi_0). $$
\end{proof}

\section{Isolation of the first eigenvalue of the fractional Lane-Emden equation}\label{sec5}

In this section, we prove 
Theorem \ref{thm:isolation}. We assume that $ \Omega$ has a Wiener regular boundary for the $s$-fractional Laplacian.

For the proof of Theorem \ref{thm:isolation}, we need the following results.
\begin{lem}\label{lem:bound}
   For $1< q < 2, \lambda >0,$ all the weak solutions $v \in \mathcal{D}^{s,2}_0(\Omega)$ of
\begin{align*}
    (-\Delta)^s v = \lambda \|v \|_{\L^q(\Omega)}^{2-q}|v|^{q-2} v, \quad \text{in } \Omega,
\end{align*}
  satisfy
  \begin{equation}
\begin{aligned}
\label{eq:boundedness}
    \|v\|_{\L^{\infty}(\Omega)} \leq  \begin{cases}
       C \lambda^{\frac{2^{\ast}_s}{2(2^{\ast}_s-q)}}\|v\|_{\L^q(\Omega)}, \quad &
    \text{if } 2s < n,\\
     C \lambda\|v\|_{\L^q(\Omega)}, \quad &\text{if } 2s \geq n,
    \end{cases}
\end{aligned}
\end{equation}
 where $C$ depends on $n,q,s,|\Omega|$. Moreover, for all $\lambda_1,\lambda_2$ and weak solutions $u,v \in \mathcal{D}^{s,2}_0(\Omega)$ of $(-\Delta)^s u= \lambda_1 |u|^{q-2} u,(-\Delta)^s v= \lambda_2 |v|^{q-2} v$ in $\Omega$, we have
\begin{align}
\label{eq:Hölder}
  \|u-v\|_{\L^{\infty}(\Omega)} \leq C \|\lambda_1 |u|^{q-2} u - \lambda_2 |v|^{q-2} v\|_{\L^{r}(\Omega)},
\end{align}
  for every $r > \frac{n}{2s}$, where $C$ is a constant depending on $n,q,s,r,\Omega$.
\end{lem}

\begin{proof}
    We refer to \cite[Prop. 3.1]{FL} for the proof of the first part. For the second part, we have $u-v \in \mathcal{D}^{s,2}_0(\Omega)$ and $(-\Delta)^s (u-v) = \lambda_1 |u|^{q-2} u - \lambda_2 |v|^{q-2} v$ weakly in $\Omega$. Hence, by \cite[Thm. 3.1]{BP}, we conclude the second part.
    
\end{proof}
\begin{prop}
For every $\lambda>0$, the weak solution $v \in \mathcal{D}^{s,2}_0(\Omega)$ of 
\[
\begin{aligned}
(-\Delta)^s v= \lambda |v|^{q-2}v, \quad &\text{in } \Omega,
\end{aligned}
\]
satisfies
\begin{equation}
\label{eq:Schauder}
\begin{aligned}
     v \in C^{2s+\epsilon}_{\loc}(\Omega)\cap \L_{2s}^1(\R^n), \quad \text{if } 2s<1,\\
    v \in C^{1,2s+\epsilon-1}_{\loc}(\Omega)\cap \L_{2s}^1(\R^n), \quad \text{if } 2s\geq 1.
\end{aligned}
\end{equation}
 for some $\epsilon >0$. In particular, $v$ is also a pointwise solution.
\end{prop}
\begin{proof}
    By Lemma \ref{lem:bound}, $v \in \L^{\infty}(\Omega).$ Hence, $\lambda |v|^{q-2} v$ is bounded in $\Omega.$ Then, using \cite[Thm. 5.4]{S}, we obtain $v \in C^{\alpha}_{\loc}(\Omega)$ for some constant $\alpha>0$. In conclusion, $\lambda |v|^{q-2} v$ is locally Hölder continuous in $\Omega$. Now, we are in the position to use the Schauder-type estimates in \cite[Section 6]{R} to arrive at \eqref{eq:Schauder}. Also, Remark \ref{Rmk:classical-weak} implies $(-\Delta)^s v = \lambda |v|^{q-2} v$ pointwise in $\Omega.$
\end{proof}

\begin{proof}[Proof of Theorem \ref{thm:isolation}]
    For the sake of contradiction, let $v_i \in \mathcal{D}^{s,2}_0(\Omega) \setminus \{u,0\}$ be a sequence of functions such that
    \begin{align*}
        \|v_i\|_{\L^q(\Omega)}&= 1,\\
(-\Delta)^s v_i &= \lambda_i |v_i|^{q-2} v_i, \quad \text{in } \Omega,\\
 \lim_{i \to \infty} \lambda_i &= \Lambda_q.
    \end{align*}
 Then, $\|v_i\|_{\H^s(\R^n)} = \lambda_i$ and by fractional Sobolev embedding, see Theorem \ref{thm:fractionalsobolev}, up to a subsequence $v_i$ converges in $\L^2(\Omega)$ to $u \in \mathcal{D}_0^{s,2}(\Omega)$, satisfying
    \begin{align*}
        (-\Delta)^s u = \Lambda_q |u|^{q-2}u, \quad \text{in } \Omega.
    \end{align*}
    By \cite[Thm. A]{FL}, up to a multiplicative constant, we can assume that $u$ is positive on $\Omega$. Now, in the light of Remark \ref{Rmk:classical-weak}, the equations $(-\Delta)^s v_i = \lambda_i |v_i|^{q-2} v_i, \, (-\Delta)^s u = \Lambda_q u^{q-1}$ hold pointwise in $\Omega$. Since $ \Omega$ satisfies Wiener criterion, $u,v_i$ belong to $C(\overline \Omega).$ Using Lemma \ref{lem:bound},
\begin{equation}
\label{eq:boundednesssequence}
\begin{aligned}
     \|v_i\|_{\L^{\infty}(\Omega)} &\leq \begin{cases}
   C \lambda_i^{\frac{2^{\ast}_s}{2(2^{\ast}_s-q)}}, \quad &
    \text{if } 2s < n,\\
     C \lambda_i, \quad &\text{if } 2s \geq n,
     \end{cases}\\
       \|u\|_{\L^{\infty}(\Omega)} &\leq \begin{cases}
   C \Lambda_q^{\frac{2^{\ast}_s}{2(2^{\ast}_s-q)}}, \quad &
    \text{if } 2s < n,\\
     C \Lambda_q, \quad &\text{if } 2s \geq n,
     \end{cases}
     \end{aligned}
\end{equation} and
    \begin{align}
    \label{eq:dominaconvergence}
        \|u-v_i\|_{\L^{\infty}(\Omega)} \leq C
 \|\, \Lambda_q u^{q-1} - \lambda_i |v_i|^{q-1} v_i\|_{\L^r(\Omega)},    \end{align}
for any $r> \frac{n}{2s}$ and a constant $C$ depending on $n,q,s,r,\Omega.$ Since $v_i$ converges to $u$ in $\L^2(\Omega)$, it converges to $u$ almost everywhere in $\Omega$. Hence, by \eqref{eq:boundednesssequence} and Lebesgue's dominated convergence, we obtain $\lim_{i \to \infty} \|\, u^{q-1} - |v_i|^{q-1} v_i\|_{\L^r(\Omega)} =0 $ for any $1\leq r < \infty.$ In conclusion, by \eqref{eq:dominaconvergence}, we imply that $v_i$ converges uniformly to $u$ on $\Omega$. By \cite[Thm. A]{FL}, the functions $v_i$ are sign-changing. Choose a sequence of points $x_i \in \overline \Omega$ such that 
\begin{align}
\label{eq:maxpoint}
\frac{1}{i}u(x_i)-v_i(x_i)= \max_{x \in \overline \Omega}{\frac{1}{i}u(x) - v_i(x)}= \max_{x \in \R^n}{\frac{1}{i}u(x) - v_i(x)}:=m_i.    
\end{align}
    Since $v_{i}$ changes sign, the maximum $m_i$ is strictly positive and $x_i \in \Omega$. Let $\tilde{x}$ be a limit point of $x_i$ in $\overline{\Omega}$. Now, we rule out the possibility of $\tilde{x} \in \Omega$. Since otherwise due to the uniform convergence
    $$\frac{1}{i}u-v_i \to -u \quad \text{uniformly on} \; \Omega,$$
    we will have
    $$0 \leq \lim_{i \to \infty }m_i = -u(\tilde{x}) <0.$$
Hence, $\tilde{x} \in \partial \Omega$.
    After passing to a subsequence, we can choose a set $K \Subset \Omega$ such that $\dist(x_i,K) > \delta$ for some $\delta >0$ and $\lim_{i \to \infty} x_i = \tilde{x} \in \partial \Omega$.
    
    Now, by \eqref{eq:maxpoint}, we get
\begin{align*}
 v_i(x_i)-v_i(y)  \leq  \frac{1}{i} (u(x_i)- u(y)),
\end{align*}
for all $y \in \R^n.$ Hence,
    \[
    \begin{aligned}
      \frac{\lambda_i}{2} |v_{i}|^{q-2}v_{i}(x_i)&=  \int_{\R^n}\frac{v_i(x_i)-v_i(y)}{|x_i-y|^{n+2s}} \d y = \int_{\R^n \setminus K}\frac{v_i(x_i)-v_i(y)}{|x_i-y|^{n+2s}} \d y + \int_{K}\frac{v_i(x_i)-v_i(y)}{|x_i-y|^{n+2s}} \d y \\
        & \leq \frac{1}{i} \int_{\R^{n} \setminus K} \frac{u(x_i)-u(y)}{|x_i-y|^{n+2s}} \d y + \int_{K}\frac{v_i(x_i)-v_i(y)}{|x_i-y|^{n+2s}} \d y \\
        &= \frac{1}{i} \int_{\R^n} \frac{u(x_i)-u(y)}{|x_i-y|^{n+2s}} \d y - \frac{1}{i}\int_{K} \frac{u(x_i)-u(y)}{|x_i-y|^{n+2s}} \d y + \int_{K} \frac{v_i(x_i)-v_i(y)}{|x_i-y|^{n+2s}} \d y \\
       & := \mathcal{I}_1(i) + \mathcal{I}_2(i) + \mathcal{I}_3(i).
    \end{aligned}
    \]
    The left hand side converges to $u^{q-1}(\tilde{x})=0$. As for the right-hand side,
    $$ \lim_{i \to \infty } \mathcal{I}_{1}(i) = \lim_{ i \to \infty } \frac{\Lambda_q}{2i} u^{q-1}(x_i)= 0,$$
    $$\lim_{i \to \infty } \mathcal{I}_{2}(i)=0,$$
    where we used $|x_i-y| >\delta $ for $y \in K$ in the second equation.
    Since $\dist(x_i,K) > \delta$ and $v_i$ converges uniformly to $u$,
    \[
    \lim_{i \to \infty } \mathcal{I}_{3}(i)= \int_{K} \frac{u(\tilde{x}) - u(y)}{|\tilde{x}- y|^{n+2s}} \d y = - \int_{K}\frac{u(y)}{|\tilde{x}-y|^{n+2s}}  \d y .
    \]
    In conclusion,
    $$0 \leq \lim_{i \to \infty} \; \mathcal{I}_{1}(i)+\mathcal{I}_{2}(i)+\mathcal{I}_{3}(i) = - \int_{K} \frac{u(y)}{|\tilde{x}-y|^{n+2s}}  \d y < 0. $$  The contradiction above concludes the proof.
    \end{proof}
    \begin{rem}\label{rmk:easyproof}
        In the case of $C^{1,1}$ domains, one can apply a simpler argument to prove Theorem \ref{thm:isolation}, using Hopf's lemma, see \cite[Lem. 1.2]{GS}. Let $v_i$ be the same sequence as in the above proof, and let $d: \Omega \to \R$ be the function $d(x) := \dist(x,\partial \Omega)$ for all $x \in \Omega$. Then, by \cite[Thm. 1.2]{RS1}, $\|\frac{v_i}{d^s}\|_{C^{\alpha}(\overline \Omega)}$ is uniformly bounded. Also, by the same argument in the proof of Theorem \ref{thm:isolation} and passing to a subsequence, $v_i$ converges pointwise to a function $u$ which is positive on $\Omega$, up a multiplicative constant.  Hence, by the Arzelà–Ascoli theorem, $\frac{v_i}{d^s}$ converges uniformly to $\frac{u}{d^s}$ on $\overline \Omega$ up to a subsequence, since $v_i$ converges pointwise to $u$. In conclusion, if $v_i$ is sign changing, then $\frac{u}{d^s}$ goes to zero at a boundary point, which is in contradiction with Hopf's lemma. In conclusion, by  \cite[Thm. A]{FL}, the proof is completed.
    \end{rem}

\section{Generalized Hopf's lemma and global boundary Harnack inequalities for viscosity solutions}\label{sec6}
In this section, we prove a generalized Hopf's lemma and a  global version of boundary Harnack inequality. We cannot work with pointwise solutions, due to the lack of enough regularity. To get rid of the issue, we use viscosity solutions instead.

To prove Hopf's lemma and a generalized boundary Harnack theorem, we need the following technical lemma, together with the strong maximum principle. The proof of the lemma is essentially the same as the argument in the proof of Theorem \ref{thm:isolation}.
\begin{lem}
\label{thm:stabilityharnack}
    Let $K \Subset \Omega$ and $ (-\Delta)^s v_i \geq f_i, (-\Delta)^s u_i \leq g_i$ in $\Omega$ in the viscosity sense, where $u_i,v_i$ belong to $\L^{1}_{2s}(\Omega) \cap C(\overline \Omega)$ and $f_i,g_i$ belong to $ C(\Omega)$. Assume that
       \begin{align*}
           v_i \to v, \quad &\text{ uniformly in every } \Omega^\prime \Subset \Omega,\\   v >0, \quad &\text{ in } \Omega,\\
\limsup_{i \to \infty} f_i(x_i)     \geq  - 2 \int_{K} \frac{v(y)}{|x-y|^{n+2s}} \d y, \quad &\text{if } \lim_{i \to \infty} x_i =x \in \partial \Omega, \, \limsup_{i \to \infty } v_i(x_i) \leq 0,
    \end{align*}
    for a function $v: \Omega \to \R$, and \begin{equation*}
\begin{aligned}
\limsup_{i \to \infty} g_i(x_i)  \leq  1, \quad &\text{if } \lim_{i \to \infty} x_i \in \partial \Omega.\\
0 \leq u_i \leq  v_i, \quad &\text{in } \R^n \setminus \Omega.
       \end{aligned}
    \end{equation*}
Then, there exist an integer $N$ and a constant $C>0$ such that 
   \begin{equation*}
      u_i \leq C (\|u_i\|_{\L^{\infty}(\Omega)}+1)  v_i, \quad \text{in } \Omega,
   \end{equation*}
 for every  $i \geq N$.
\end{lem}
\begin{proof}
    Without loss of generality, by taking $K$ a bit larger and using $v>0$ in $\Omega$, we can assume that if $\limsup_{i \to \infty } v_i(x_i) \leq 0$ and $\lim_{i \to \infty} x_i =x \in \partial \Omega \, $, then the following inequality is strict $$\limsup_{i \to \infty}f_i(x_i) > -2  \int_{K} \frac{v(y)}{|x-y|^{n+2s}} \d y.$$ Also, by normalization $\frac{u_i}{\|u_i\|_{\L^{\infty}(\Omega)}+1}$, we can assume that $\|u_i\|_{\L^{\infty}(\Omega)} \leq 1$. After doing this normalization for $u_i$, the assumptions $0\leq u_i \leq v_i$ still remain true. For the sake of contradiction, we choose a sequence of points $x_i \in \overline \Omega$ and $C_i>1$ such that $\lim_{i \to \infty} C_i = \infty$ and 
    \begin{align}
    \label{eq:points}
    \frac{1}{C_i}u_i(x_i)-v_i(x_i) = \max_{\overline \Omega} \frac{1}{C_i}u_i-v_i >0.  
    \end{align}
    Note that we used the fact that the continuous function $\frac{1}{C_i}u_i-v_i$ takes its maximum at a point in $\overline \Omega$. Since $u_i|_{\partial \Omega} \leq  v_{i}|_{\partial \Omega}$, we have $x_i \in \Omega$. Let $\tilde{x}$ be a limit point of $x_i$ in $\overline{\Omega}$. Now, we rule out the possibility of $\tilde{x} \in \Omega$. Since otherwise due to the uniform convergence
    $$\frac{1}{C_i}u_i(x)-v_i(x) \to -v(x), \quad \text{uniformly in every }  \Omega^\prime \Subset \Omega,$$
    we have the following contradiction
    $$0 \leq \lim_{i \to \infty }\frac{1}{C_i}u_i(x_i)-v_i(x_i) = -v(\tilde{x}) <0.$$
   Hence, $\tilde{x} \in \partial \Omega$. Note that by \cite[Lem. 5.8]{CS2}, we have
\begin{equation}
\label{eq:viscosity}
     (-\Delta)^s\biggl(\frac{1}{C_i} u_i- v_i\biggr)(x)   \leq \frac{g_i(x)}{C_i} - f_i(x), \quad \text{ in } \Omega ,
\end{equation}
in the viscosity sense. Now, consider a sequence $ 0<r_i < \frac{1}{2^i}, \phi_i \in C^2( B(x_i,r_i))$, such that $\overline {B(x_i,r_i)} \subset \Omega$, where
\begin{align*}
    &\phi_i(x) := |x-x_i|^2 + \frac{1}{C_i}u_i(x_i) - v_i(x_i) \quad \text{ in } B(x_i,r_i) \setminus \{x_i\},
\end{align*}
 and define the functions \begin{align}
 \label{eq:touchabove}
    &w_i := \begin{cases}
        \phi_i \quad &\text{ in }  B(x_i,r_i),\\
        \frac{1}{C_i} u_i- v_i \quad &\text{ in } \R^n \setminus  B(x_i,r_i).
    \end{cases}
\end{align} After passing to a subsequence and taking $r_i$ small enough, depending on $K$, we have  $    B(x_i,2r_i) \subset \Omega \setminus K$ for every $i$ and $x_i \to \tilde{x} \in \partial \Omega$ as $i \to \infty$.
Since $r_i \to 0$ and $x_i \to \tilde x \in \partial \Omega$ as $i \to \infty$, we have the uniform convergence $ w_i \to -v$ on every $K \Subset \Omega$ as $i \to \infty$ and $\dist(x_i,K) > \delta$ for a $\delta>r_i$ and every $i$. Also, by
    $ v_i(x_i) <\frac{1}{C_i}  u_i(x_i)$, we have
    ´\begin{align}
    \label{eq:limsup}
        \limsup_{i \to \infty}  v_i(x_i) \leq 0.
    \end{align} 
    Now, by \eqref{eq:viscosity} and the fact that $\phi_i$ touches $\frac{1}{C_i} u_i -v_i$ at $x=x_i$ from above, it follows that
    \[
    \begin{aligned}
 & \limsup_{i \to \infty} f_i(x_i) \leq \limsup_{i \to \infty} 2 \int_{\R^n} \frac{w_i(y)-w_i(x_i)}{|y-x_i|^{n+2s}}\, \d y \\& \leq  \limsup_{i \to \infty} 2 \int_{B(x_i,r_i)} \frac{\phi(y)-\phi_i(x_i)}{|y-x_i|^{n+2s}}
  \, \d y \\& + \limsup_{i \to \infty}
2\int_{\R^n \setminus (K \cup B(x_i,r_i) )} \frac{w_i(y)-w_i(x_i)}{|x_i-y|^{n+2s}} \, \d y
   \\&+ \limsup_{i \to \infty} 2
\int_{K }   \frac{w_i(y)-w_i(x_i)}{|y-x_i|^{n+2s}}\, \d y = \mathcal{I}_{1} + \mathcal{I}_{2}+ \mathcal{I}_{3}.
    \end{aligned}
    \]
 The left hand side satisfies $\limsup_{i \to \infty} f_i(x_i)> - 2\int_{K} \frac{v(y)}{|\tilde{x}-y|^{n+2s}} \d y$. Note that
 
 \begin{align}\label{eq:negative}
     w_i(y)- w_i(x_i)   = \frac{1}{C_i}u_i(y)-v_i(y) -\biggl(\frac{1}{C_i}u_i(x_i)-v_i(x_i)\biggr) \leq 0,
 \end{align}
  for $y \in \R^n \setminus B(x_i,r_i)$ by definition of $w_i$,  \eqref{eq:points}, and $C_i >B.$ Hence, \begin{align*}
     \mathcal{I}_{2} \leq 0.
 \end{align*}
 Also,
    $$  \mathcal{I}_{1} =   \limsup_{i \to \infty } 2 \int_{B(x_i,r_i)} \frac{|y-x_i|^2}{|y-x_i|^{n+2s}} \, \d y = \limsup_{i \to \infty } 2 r_i^{(2-2s)}=0,$$
    and since $|x_i-y| > \delta,$ $w_i$ converges uniformly to $-v$ on $K$, and \eqref{eq:limsup}, it is implied that
    \begin{align*}
         \mathcal{I}_{3}  & \leq 2\int_{K} \frac{ \limsup_{i \to \infty}  v_i(x_i) -  v(y)}{|\tilde{x}- y|^{n+2s}} \d y \\& \leq  -2 \int_{K}\frac{v(y)}{|\tilde{x}-y|^{n+2s}}  \d y,
    \end{align*}   by \eqref{eq:negative}, Fatou's lemma, and \eqref{eq:limsup}.
    In conclusion, we arrive at 
    $$-2 \int_{K} \frac{v(y)}{|\tilde{x}-y|^{n+2s}} \d y  < \mathcal{I}_{1}+\mathcal{I}_{2}+\mathcal{I}_{3} \leq -2  \int_{K} \frac{v(y)}{|\tilde{x}-y|^{n+2s}}  \d y. $$
 Hence, we get a contradiction.
\end{proof}

\begin{defn}
    Define the torsion function $u_{\tor} \in \L^1_{2s}(\Omega) \cap C(\overline{\Omega})$ by the following properties:
\begin{align*}
u_{\tor}=0, \quad &\text{in } \R^n \setminus \Omega,\\
    (-\Delta)^s u_{\tor} = 1,\quad &\text{in } \Omega,
\end{align*}
in the viscosity sense, see Lemma \ref{lem:existenceweaksol} and Proposition \ref{prop:equaivweakviscos} for the existence if $\Omega$ has a regular boundary for the $s$-fractional Laplacian.
\end{defn}
\begin{lem}
\label{lem:maxprinciple}
    Let $u \in \L^{1}_{2s}(\Omega)$ be a viscosity supersolution of $(-\Delta)^s u \geq f$ in $\Omega$ . Assume that there is a point $x_0 \in \Omega$ such that $u(x_0) = \inf_{\R^n} u$, then
    \begin{align*}
         2 \int_{\R^n} \frac{u(x_0)-u(y)}{|x_0-y|^{n+2s}} \, \d y \geq f(x_0).
    \end{align*}
\end{lem}
\begin{proof}
    Notice that the integral on the left-hand side makes sense without $\mathrm{P.V.}$ since we have $u(x_0)-u(y) \leq 0$, although the integral might be $- \infty$. Define $u_{\epsilon}(y)= u(x_0)-|x_0-y|^2$ for $y \in B(x_0,\epsilon)$ and $u_{\epsilon}= u$ in $\R^n \setminus B(x_0,\epsilon)$ for $\epsilon>0$ small enough, such that $B(x_0,\epsilon) \subset \Omega$. Then, $ u_{\epsilon} \leq u$ in $\R^n$, $u_{\epsilon} \in C^{2}(B(x_0,\epsilon))$, and $u_{\epsilon}(x_0) = u(x_0)$. Hence,
    \begin{align*}
 &2 |B(0,1)|\epsilon^{2-2s}+ 2 \int_{\R^n
 \setminus B(x_0,\epsilon)}  \frac{u(x_0)-u(y)}{|x_0-y|^{n+2s}} \, \d y\\&= 2\int_{ B(x_0,\epsilon)}  \frac{|x_0-y|^2}{|x_0-y|^{n+2s}}\, \d y+ 2\int_{\R^n \setminus B(x_0,\epsilon)}  \frac{u_{\epsilon}(x_0)-u_{\epsilon}(y)}{|x_0-y|^{n+2s}} \, \d y \\ &= 2\int_{\R^n} \frac{u_{\epsilon}(x_0)-u_{\epsilon}(y)}{|x_0-y|^{n+2s}}  \, \d y  \geq f(x_0).
\end{align*}
Letting $\epsilon \to 0$ and using the monotone convergence theorem, it is obtained that
\begin{align*}
 f(x_0) \leq  2 \int_{\R^n} \frac{u(x_0)-u(y)}{|x_0-y|^{n+2s}} \, \d y.
\end{align*}
\end{proof}
Now, we prove Hopf's lemma. 
\begin{proof}[Proof of Lemma \ref{lem:Hopf's lemma}]
    The assumption 
    \[
    f(x_0) > -2 \int_{\R^n} \frac{u(y)}{|x_0-y|^{n+2s}} \, \d y, \quad \text{if } x_0 \in \Omega, \, u(x_0)=0,
    \]
    ensures that $u> 0$ in $\Omega$, using Lemma \ref{lem:maxprinciple}. Hence, by the assumption
    $$\limsup_{x \to x_0} f(x) \geq -2\int_{K} \frac{u(y)}{|x_0 - y|^{n+2s}} \, \d y, \quad  \text{if } x_0 \in \partial \Omega, \lim_{x \to x_0} u(x) =0,$$
    we can apply Lemma \ref{thm:stabilityharnack} to the sequences $u_i=u_{\tor},v_i =u$ to obtain $u \geq C u_{\tor}$ for a constant $C >0.$
\end{proof}

The following result is another version of Hopf's Lemma for the viscosity supersolutions.
\begin{lem}
\label{lem:weakhopf}
    Let $u \in \L^1_{2s}(\R^n) \cap C(\overline \Omega)$, satisfying $u\geq 0$ in $\R^n \setminus \Omega$ and $$(-\Delta)^s u  \geq - 2 (\mathrm{diam}\, \Omega)^{-(n+2s)}\int_{K} u(y)\, \d y, \quad \text{ in the viscosity sense in } \Omega,$$ where $K \Subset \Omega$. Then, either $u=0$ a.e. in $\R^n \setminus K$ or $u>0$ in $\Omega$ and $u \geq C u_{\tor}$ in $\Omega$ for a constant $C>0.$
\end{lem}
\begin{proof}  
Assume that there exists a point $x_0 \in \Omega$, satisfying  $u(x_0)= \inf_{\R^n} u \leq 0$. Then, by Lemma \ref{lem:maxprinciple}, 
\begin{align*}
  2\int_{\R^n} \frac{u(x_0)-u(y)}{|x_0-y|^{n+2s}} \, \d y &\geq - 2 (\mathrm{diam}\, \Omega)^{-(n+2s)}\int_{K} u(y)\, \d y \\
 &\geq 2 (\mathrm{diam}\, \Omega)^{-(n+2s)}\int_{K} (u(x_0)- u(y))\, \d y \\& \geq 2\int_{K} \frac{u(x_0)- u(y)}{|x_0-y|^{n+2s}} \, \d y. 
\end{align*}
In conclusion, $u(y) = u(x_0) \leq 0$ a.e. in $\R^n \setminus K $. Combining this result with the assumption $u \geq 0$ in $\R^n \setminus \Omega$, we derive $u=0$ a.e. in $\R^n \setminus K.$ 
  Finally, let $u > 0$ in $\Omega$. Then,
 $$(-\Delta)^s u \geq - 2 (\mathrm{diam}\, \Omega)^{-(n+2s)}\int_{K} u(y)\, \d y \geq - 2\int_{K} \frac{ u(y)}{|x_0-y|^{n+2s}} \, \d y, \quad \text{in } \Omega,$$  in the viscosity sense.
 Hence, by Lemma \ref{thm:stabilityharnack}, we conclude that $u \geq C u_{\tor}$ in $\Omega$ for a constant $C>0$. 
\end{proof}
Now, we aim at proving Theorem \ref{thm:boundaryharnack}. First, we need the following result.
\begin{cor}
\label{Cor:boundaryharnack}
    Let $u,v$ belong to $\L^{\infty}(\R^n) \cap C(\overline \Omega)$ and satisfy 
\begin{equation*}  
\begin{aligned}
     u \geq 0 ,\, v > 0,   \quad & \text{in } \Omega,\\
0 \leq u  \leq  v, \quad & \text{in } \R^n \setminus \Omega,
\end{aligned}
\end{equation*}
and \begin{equation*}
     (-\Delta)^s u \leq 1 ,\, (-\Delta)^s v \geq f, \quad  \text{in } \Omega,
\end{equation*}
 in the viscosity sense for a function $f\in C(\Omega)$. Assume, for $K \Subset \Omega$, we have
      \begin{align*}
      \limsup_{i \to \infty} f(x_i)     \geq  - 2 \int_{K} \frac{v(y)}{|x-y|^{n+2s}} \d y, \quad \text{if } \lim_{i \to \infty} x_i =x \in \partial \Omega, \, \lim_{i \to \infty } v(x_i) = 0,
    \end{align*}
   at every point $x$ on the boundary. Then,
\begin{equation*}
    u \leq C v, \quad \text{in } \Omega,
\end{equation*}
where the constant $C$ is independent of $u$ but dependent on $v$. 
\end{cor}
\begin{proof}
    We argue by contradiction. If there is no such universal constant $C$, then we can pick a sequence of non-negative functions $u_i \in \L^{\infty}(\R^n) \cap C(\overline \Omega)$ and constants $C_i$ diverging to infinity, such that 
\begin{equation}
\begin{aligned}\label{eq:sub-sol-uniform}
 \sup_{\Omega} \frac{u_i}{v} > C_i,&\\ 
(-\Delta)^s u_i \leq 1,& \quad \text{in } \Omega, \\
0 \leq u_i \leq  v,& \quad \text{in } \R^n \setminus \Omega.
\end{aligned}
\end{equation}
Hence, $0\leq u_i \leq  u_{\tor}+ \|v\|_{\L^{\infty}(\R^{n}\setminus \Omega)}$ in $\Omega$, by \eqref{eq:sub-sol-uniform} and \cite[Cor. 4.6]{CS2}. In conclusion, the sequence $u_i$ is uniformly bounded in $\Omega$. Thus, we are in the position to apply Lemma \ref{thm:stabilityharnack} for sequences $v_i=v$ and $u_i$ the same as here. This gives us $u_i \leq C v$ which leads to a contradiction.
\end{proof}

\begin{proof}[Proof of Theorem \ref{thm:boundaryharnack}]
Assume for the sake of contradiction that there are no such constants $C_1$ and $C_2$. That is there exist sequences $u_i$ and $v_i$ with $u_i(x_0) \geq D, v_i(x_0) \geq D$ ($\|u_i\|_{\L^p(\Omega \setminus K)} \geq D, \|v_i\|_{\L^p(\Omega \setminus K)} \geq D$), satisfying the conditions of the Theorem \ref{thm:boundaryharnack}, that is
\begin{align*}
    u_i>0 , \,v_i>0 \quad & \text{in } 
\Omega,\\
 0\leq v_i = u_i=g_i \leq 1 \quad & \text{in }  \R^n \setminus \Omega,
    \end{align*}
    \begin{align*}
   - 2 (\mathrm{diam}\, \Omega)^{-(n+2s)}\int_{K} u_i(y)\, \d y \leq (-\Delta)^s u_i \leq  1, \quad & \text{in } \Omega,\\
    -2 (\mathrm{diam}\, \Omega)^{-(n+2s)}\int_{K} v_i(y) \, \d y  \leq   (-\Delta)^s v_i \leq 1, \quad &\text{in } \Omega,
    \end{align*}
such that either $\lim_{i \to \infty} \inf_{\Omega} \frac{v_i}{u_i} = 0$ or $\lim_{i \to \infty} \sup_{\Omega} \frac{v_i}{u_i}=\infty$. Then, by \cite[Cor. 4.6]{R}, 
\begin{align}
\label{eq:boundabove}
   0 \leq u_{i} \leq u_{\tor}+\|g_i\|_{\L^{\infty}(\R^n \setminus \Omega)}, \,
    0\leq v_i \leq u_{\tor}+\|g_i\|_{\L^{\infty}(\R^n \setminus \Omega)}, \quad \text{in } \R^n.
\end{align}
and, by interior Hölder regularity, see \cite[Thm. 2.1]{CS2}, and the Arzelà–Ascoli theorem, up to a subsequence, we can assume that both $v_i$ and $u_i$ converge uniformly on every compact subset of $\Omega$ to non-negative limits, say 
\[
u_i \to u \in C(\Omega) \qquad \text{and} \qquad v_i \to v \in C(\Omega).
\]
Using, \eqref{eq:boundabove}, $u_i$ and $v_i$ are uniformly bounded and Lebesgue dominated convergence implies
\begin{align*}
   \lim_{i \to \infty} \|u_i- u\|_{\L^p(\Omega)} + \|v_i- v\|_{\L^p(\Omega)} = 0.
\end{align*}
 Note that the continuity of the limits on $\Omega$ is a consequence of local uniform convergence. Now, define the auxiliary functions $\tilde{u}_i$ as 
 \[
 \tilde{u}_i(x)=\begin{cases}
     u_i(x), \quad &\text{if} \quad x \in \Omega, \\
      0, \quad &\text{if} \quad x \in \R^n \setminus \Omega,
 \end{cases}
 \]
 and define $\tilde{v}_i$ in a similar way. As $u_i$ and $v_i$ are nonnegative, $\tilde{u}_i$ and $\tilde{v}_i$ satisfy
 \[
 \begin{aligned}
    (-\Delta)^s \tilde{v}_i &\geq   - 2(\mathrm{diam}\, \Omega)^{-(n+2s)}\int_{K} v_i(y) \, \d y, \quad  \text{in } \Omega, \\
    (-\Delta)^s \tilde{u}_i &\geq   - 2(\mathrm{diam}\, \Omega)^{-(n+2s)}\int_{K} u_i(y) \, \d y, \quad  \text{in } \Omega.
 \end{aligned}
 \]
 Notice that now $\tilde{u}_i$ and $\tilde{v}_i$, converge almost everywhere in $\R^n$. 
 In light of local uniform convergence in $\Omega$, using the stability property of weak solutions \cite[Lem. 4.5]{CS2}, we obtain \begin{align*}
\label{eq:positive}
    (-\Delta)^s v &\geq   - 2(\mathrm{diam}\, \Omega)^{-(n+2s)}\int_{K} v(y) \, \d y, \quad  \text{in } \Omega, \\
    (-\Delta)^s u &\geq   - 2(\mathrm{diam}\, \Omega)^{-(n+2s)}\int_{K} u(y) \, \d y, \quad  \text{in } \Omega,
\end{align*}
where $u$ and $v$ are extended to be $0$ outside of $\Omega$.
Moreover, we have
$$u(x_0) \geq D, \quad v(x_0)\geq D \; (\|u\|_{\L^p(\Omega \setminus K)} \geq D, \|v\|_{\L^p(\Omega \setminus K)} \geq D).$$
Therefore, $u$ and $v$ are not identically zero in $\Omega \setminus K$. Hence, using Lemma \ref{lem:maxprinciple}, we arrive at 
$$v>0 \quad \text{in }\Omega \qquad  \text{and} \qquad u>0 \quad \text{in } \Omega.$$
As 
\[
\lim_{i \to \infty}  - 2(\mathrm{diam}\, \Omega)^{-(n+2s)}\int_{K} v_i(y) \, \d y =- 2\int_{K} \frac{v(y)}{(\mathrm{diam}\, \Omega)^{n+2s}} \, \d y \geq -2 \int_{K} \frac{v(y)}{|x-y|^{n+2s}} \d y,
\]
  and $(-\Delta)^s u_i \leq 1$ in $\Omega$ in the viscosity sense, the sequence $v_i$ satisfies the assumptions of Lemma \ref{thm:stabilityharnack}. In conclusion, $\frac{1}{C} \leq \frac{v_i}{u_i}$ in $\Omega$ for some constant $C>0$ and large enough $i$, and $\lim_{i \to \infty} \inf_{\Omega} \frac{v_i}{u_i}=0$ cannot occur.
 One can similarly verify the conditions
\[
\lim_{i \to \infty}  - 2(\mathrm{diam}\, \Omega)^{-(n+2s)}\int_{K} u_i(y) \, \d y =- 2\int_{K} \frac{v(y)}{(\mathrm{diam}\, \Omega)^{n+2s}} \, \d y \geq -2 \int_{K} \frac{u(y)}{|x-y|^{n+2s}} \d y,
\]
  and $(-\Delta)^s v_i \leq 1$ in $\Omega$ in the viscosity sense. Using Lemma \ref{thm:stabilityharnack}, with the role of $u_i$ and $v_i$ reversed, we arrive at $C \leq \frac{u_i}{v_i}$ in $\Omega$ for some constant $C>0$ and large enough $i$. Hence, the assumption $\lim_{i \to \infty} \sup \frac{v_i}{u_i}= \infty$ cannot occur.
\end{proof}

The set $K$ in Theorem \ref{thm:boundaryharnack} can be of measure zero or even empty in the limiting case. For an empty set $K$, the statement of the theorem reads as follows:

\begin{cor}
   Let $u,v \in C(\overline \Omega) \cap \L^{\infty}(\R^n)$ satisfy  
    \begin{align*}
    u>0 , \,v>0 \quad & \text{in } \; \Omega,\\
0\leq  u = v \leq 1  \quad & \text{in }  \R^n \setminus \Omega,
    \end{align*}
  and
    \begin{align*}
    0 \leq (-\Delta)^s u \leq  1, \quad & \text{in } \Omega,\\
    0  \leq   (-\Delta)^s v \leq 1, \quad &\text{ in } \Omega,
    \end{align*}
    in the viscosity sense. Assume that either $u(x_0)\geq D , \; v(x_0)\geq D$ for a fixed point $x_0 \in \Omega $ or $\|u\|_{\L^p(\Omega)} \geq D , \;  \|v\|_{\L^p(\Omega )}\geq D$ for $D>0, 1 \leq p < \infty.$ Then,
    \begin{align*}
       C_1 \leq \frac{u}{v} \leq C_2, \quad \text{in } \Omega,
    \end{align*} where $C_1, C_2$ are positive constants depending on $\Omega, n, s, D, x_0$ or $p$. 
\end{cor}

\begin{rem}
    By Proposition \ref{prop:equaivweakviscos}, Lemma \ref{lem:Hopf's lemma} holds if in addition to the assumptions of the Lemma we have $u \in \H^{s}_{\loc}(\Omega)$ and $(-\Delta)^s u\geq f$ locally weakly in $\Omega$. Moreover, Theorem \ref{thm:boundaryharnack} holds also if in addition we assume $u ,v  \in \H^{s}_{\loc}(\Omega)$ and \begin{align*}
  & - 2 (\mathrm{diam}\, \Omega)^{-(n+2s)}\int_{K} u(y)\, \d y \leq (-\Delta)^s u \leq  1,\\
   & -2 (\mathrm{diam}\, \Omega)^{-(n+2s)}\int_{K} v(y) \, \d y  \leq   (-\Delta)^s v \leq 1,
    \end{align*}
  locally weakly in $\Omega$.
    \end{rem}

\def\cprime{$'$} \def\cprime{$'$} \def\cprime{$'$}

\end{document}